\documentclass[reqno]{amsart}
\usepackage{amssymb}
\usepackage{hyperref}

\begin{document}
\title[Extension of Mittag-Leffler function]
{Extension of Mittag-Leffler function }
\author[G. Rahman, K. S. Nisar, S. Mubeen, M. Arshad]
{G. Rahman, K. S. Nisar, S. Mubeen, M. Arshad}  % in alphabetical order
\address{Department of Mathematics University of Sargodha, Sargodha, Pakistan}
\address{Gauhar Rahman \newline
 Department of Mathematics, International Islamic University,
Islamabad, Pakistan}
\email{gauhar55uom@gmail.com}
\address{Kottakkaran Sooppy Nisar \newline Department of Mathematics, College of Arts and
Science, Prince Sattam bin Abdulaziz University, Wadi Al dawaser, Riyadh
region 11991, Saudi Arabia}
\email{ksnisar1@gmail.com}

\address{Shahid Mubeen \newline
Department of Mathematics, University of Sargodha, Sargodha, Pakistan}
\email{smjhanda@gmail.com}
\address{ Muhammad Arshad \newline
Department of Mathematics, International Islamic University,
Islamabad, Pakistan}
\email{marshad$_{-}$zia@yahoo.com}


\begin{abstract}
In this paper, we present extension of Mittag-Leffler function by using extension of beta functions (\"{O}zergin et al. in J. Comput. Appl. Math. 235 (2011), 4601-4610) and obtain some integral representation of this newly defined function. Also we present the Mellin transform of this function in terms of Wright hypergeometric function. Furthermore, we show that the extended fractional derivative of the usual Mittag-Leffler function gives the extension of Mittag-Leffler function.
 \end{abstract}

\maketitle

\numberwithin{equation}{section}
\newtheorem{theorem}{Theorem}[section]
\newtheorem{lemma}[theorem]{Lemma}
\newtheorem{proposition}[theorem]{Proposition}
\newtheorem{corollary}[theorem]{Corollary}
\newtheorem*{remark}{Remark}
\newtheorem{definition}{Definition}
\vskip 6mm
\section{\bf\large  Introduction and Preliminaries }
\vskip 6mm
The Mittag-Leffler function occurs naturally in the solution of fractional order and integral equation. The importance of such functions in physics and engineering is steadily increasing. Some application of the Mittag-Leffler is carried out in the Study of Kinetic Equation, Study of Lorenz System, Random Walk, Levy Flights and Complex System and also in applied problems such as fluid flow, electric network, probability and statistical distribution theory. \\
We begin with the Gosta and Wiman Mittag-Leffler functions $E_\rho(z)$ and $E_{\rho,\sigma}(z)$ are defined by the following series as:
\begin{eqnarray}\label{a1}
E_\rho(z)=\sum\limits_{n=0}^{\infty}\frac{z^n}{\Gamma(\rho n+1)}, z\in\mathbb{C}; \mathfrak{R}(\rho)>0
\end{eqnarray}
and
\begin{eqnarray}\label{a2}
E_{\rho,\sigma}(z)=\sum\limits_{n=0}^{\infty}\frac{z^n}{\Gamma(\rho n+\sigma)}, z, \sigma\in\mathbb{C}; \mathfrak{R}(\rho)>0,
\end{eqnarray}
respectively.  For further study of $E_\rho(z)$ and $E_{\rho,\sigma}(z)$ such as generalizations and applications , the readers may refer to the recent work of researchers \cite{R1,R11,R4,R5,R7,R13} and the work of Saigo and Kilbas \cite{R17}. In recent years, the  function defined in (\ref{a1}) and some of generalizations have been numerically considered in the complex plane (see \cite{R10,R19}). Prabhakar \cite{R16} have introduced a generalization of the  function $E_{\rho, \sigma}(z)$ defined in (\ref{a2}) as follows:
\begin{eqnarray}\label{a3}
E_{\rho,\sigma}^{\gamma}(z)=\sum\limits_{n=0}^{\infty}\frac{(\delta)_n}{\Gamma(\rho n+\sigma)}\frac{z^n}{n!}, z, \sigma\in\mathbb{C}; \mathfrak{R}(\rho)>0,
\end{eqnarray}
where $(\delta)_n$ denote the well known Pochhammer Symbol which is defined by:
\begin{eqnarray*}
(\delta)_n=\left\{\begin{array}{l}\label{1}
1, (n=0, \delta\in\mathbb{C})\\
               \delta(\delta+1)\cdots(\delta+n-1), (n\in\mathbb{N}, \delta\in\mathbb{C}).\\
\end{array}\right.
\end{eqnarray*}
Obviously, the following special cases are satisfied:
\begin{eqnarray}
E_{\rho,\sigma}^{1}(z)=E_{\rho, \sigma}(z)=E_{\rho, 1}^{1}(z)=E_{\rho}(z).
\end{eqnarray}
In recent times many researchers have investigated the importance and great consideration of Mittag-Leffler function in the theory of special functions for exploring the generalization and some applications. Many extensions for these functions are found in (\cite{R6}, \cite{R21}-\cite{R24}). Shukla and Prajapati \cite{CP} (see also \cite{SZ}) defined and investigated the function  $E_{\rho, \sigma}^{\delta,q}(z)$, which is defined as:
\begin{eqnarray}\label{a4}
E_{\rho, \sigma}^{\delta,q}(z)=\sum\limits_{n=0}^{\infty}\frac{(\delta)_{nq}}{\Gamma(\rho n+\sigma)}\frac{z^n}{n!},
\end{eqnarray}
where $z,\sigma,\delta\in\mathbb{C}$; $\mathfrak{R}(\rho)>0$; $q>0$.
In the same paper, they  have used the well-known  Riemann-Liouville right-sided fractional integral, derivative and generalized Riemann-Liouville derivative operators (see \cite{R8,R9,R13,R18}). Very recently \"{O}zarslan and Yilmaz \cite{OY} have investigated an extended Mittag-Leffler function $E_{\rho,\sigma}^{\delta;c}(z;p)$, which is defined as:
\begin{eqnarray}\label{a5}
E_{\rho,\sigma}^{\delta;c}(z;p)=\sum\limits_{n=0}^{\infty}\frac{B_p(\delta+n, c-\delta)}{B(\delta, c-\delta)}\frac{(c)_n}{\Gamma(\rho n+\sigma)}\frac{z^n}{n!},
\end{eqnarray}
where $p\geq0$, $\mathfrak{Re}(c)>\mathfrak{Re}(\delta)>0$ and $B_p(x,y)$ is extended beta function defined in \cite{CQ,FM} as follows:
\begin{eqnarray}\label{dd}
B_p(x,y)=\int\limits_{0}^{1}t^{x-1}(1-t)^{y-1}e^{-\frac{p}{t(1-t)}}dt,
\end{eqnarray}

where $\mathfrak{Re}(p)>0$, $\mathfrak{Re}(x)>0$ and $\mathfrak{Re}(y)>0$.\\
In this paper, we define further extension of Mittag-Leffler function as:
\begin{eqnarray}\label{Fextended}
E_{\alpha, \beta}^{\gamma, c;\lambda, \rho}(z;p)=\sum\limits_{n=0}^{\infty}\frac{B_p^{\lambda,\rho}(\gamma+n, c-\gamma)}{B(\gamma,c-\gamma)}\frac{(c)_n}{\Gamma(\alpha n+\beta)}\frac{z^n}{n!}
\end{eqnarray}
where $p\geq0$, $\Re(c)>\Re(\gamma)>0$ and $B_p^{\lambda,\rho}$ is extension of extended beta function defined by
\begin{eqnarray}\label{Ebeta}
B_p^{\lambda,\rho}(x,y)=\int_0^1t^{x-1}(1-t)^{y-1}._1F_1\Big[\lambda;\rho;-\frac{p}{t(1-t)}\Big]dt,
\end{eqnarray}
where $\Re(p)>0$, $\Re(x)>0$, $\Re(y)>0$, $\Re(\lambda)>0$, $\Re(\rho)>0$. It is obvious that $B_p^{\lambda,\lambda}(x,y)=B_p(x,y)$ and $B_0(x,y)=B(x,y)$. For further details the readers are refer to the work of \"{O}zergin et al. \cite{Ozergin}.
\begin{remark}
(1) If setting $\rho=\lambda$ in (\ref{Fextended}), then it reduces to the well known extended Mittag-Leffler function as defined in (\ref{a5}).\\
(ii) If setting $\rho=\lambda$ and $p=0$ in (\ref{Fextended}), then it reduces to the well-known Mittag-Leffler function as defined in (\ref{a3}).
\end{remark}
%%%%%%%%%%%%%%%%%%%%%%%%%%%%%%%%%%%%%%%%%%%%%%%%%%%%%%%%%%%%%%%%%%%%%%%%%%%%%%%%%%%
\section{Properties of further extended Mittag-Leffler function}
We start with the following theorem, which gives the integral representation of the extended Mittag-Leffler function.
\begin{theorem}\label{tha}
Let $c, \alpha, \beta, \gamma, \lambda, \rho\in\mathbb{C}$, $\Re(c)>\Re(\gamma)>0$, $\Re(\alpha)>0$, $\Re(\beta)>0$, $\Re(\lambda)>0$, $\Re(\rho)>0$. Then for the extended Mittag-Leffler function, we have the following integral representation
\begin{eqnarray}\label{integral}
E_{\alpha,\beta}^{\gamma,c;\lambda,\rho}(z;p)=\frac{1}{B(\gamma,c-\gamma)}\int_0^1t^{\gamma-1}(1-t)^{c-\gamma-1}
._1F_1\Big[\lambda;\rho;-\frac{p}{t(1-t)}\Big]E_{\alpha,\beta}^{c}(tz)dt.
\end{eqnarray}
\end{theorem}
\begin{proof}
Using equation (\ref{Ebeta}) in equation (\ref{Fextended}), we have
\begin{align*}
E_{\alpha,\beta}^{\gamma,c;\lambda,\rho}(z;p)&=\sum\limits_{n=0}^{\infty}\Big\{\int_0^1t^{\gamma+n-1}(1-t)^{c-\gamma-1}\\
&\times{}_1F_1\Big[\lambda;\rho;-\frac{p}{t(1-t)}\Big]dt\Big\}\frac{(c)_n}{B(\gamma,c-\gamma)}\frac{z^n}{\Gamma(\alpha n+\beta)n!}\Big\}.
\end{align*}
Interchanging the order of summation and integration in above equation, we get
\begin{align*}
E_{\alpha,\beta}^{\gamma,c;\lambda,\rho}(z;p)&=\int_0^1t^{\gamma-1}(1-t)^{c-\gamma-1}\\
&\times{}_1F_1\Big[\lambda;\rho;-\frac{p}{t(1-t)}\Big]\sum\limits_{n=0}^{\infty}
\frac{(c)_n}{B(\gamma,c-\gamma)}\frac{(tz)^n}{\Gamma(\alpha n+\beta)n!}dt.
\end{align*}
Using equation (\ref{a3}) in above equation, we get the desired integral representation.
\end{proof}
\begin{corollary}
Substituting $t=\frac{u}{1+u}$ in Theorem \ref{tha}, we get
\begin{eqnarray}
E_{\alpha,\beta}^{\gamma,c;\lambda,\rho}(z;p)=\frac{1}{B(\gamma,c-\gamma)}\int_0^\infty\frac{u^{\gamma-1}}{(u+1)^c}._1F_1
\Big[\lambda;\rho;-\frac{p(1+u)^2}{u}\Big]E_{\alpha,\beta}^{c}\Big(\frac{uz}{1+u}\Big)du.
\end{eqnarray}
\end{corollary}
\begin{corollary}
Taking $t=\sin^2\theta$ in Theorem \ref{tha}, we get the following integral representation
\begin{align}
E_{\alpha,\beta}^{\gamma,c;\lambda,\rho}(z;p)&=\frac{1}{B(\gamma,c-\gamma)}\Big[2\int_0^\infty\sin^{2\gamma-1}\theta\cos^{2c-1}\theta._1F_1
\Big[\lambda;\rho;-\frac{p}{\sin^2\theta\cos^2\theta}\Big]\Big]\notag\\
&\times E_{\alpha,\beta}^{c}\Big(z\sin^2\theta\Big)d\theta.
\end{align}
\end{corollary}
  Kurulay and Bayram \cite{Kurulay} introduced the following recurrence relation  for Prabhakar Mittag-Leffler function as:
$$E_{\alpha,\beta}^{c}(tz)=\beta E_{\alpha,\beta+1}^{c}(tz)+\alpha z\frac{d}{dz} E_{\alpha,\beta+1}^{c}(tz).$$
Inserting the above recurrence relation into (\ref{integral}), we get the following recurrence relation for the newly defined extended Mittag-Leffler function.
\begin{corollary}
Let $p\geq0$, $\Re(c)>\Re(\gamma)>0$, $\Re(\alpha)>0$, $\Re(\beta)>0$, $\Re(\lambda)>0$, $\Re(\rho)>0$, then the following relation holds:
\begin{eqnarray}
E_{\alpha,\beta}^{\gamma,c;\lambda,\rho}(z;p)=\beta E_{\alpha,\beta+1}^{\gamma,c;\lambda,\rho}(z;p)+\alpha z\frac{d}{dz}E_{\alpha,\beta+1}^{\gamma,c;\lambda,\rho}(z;p)
\end{eqnarray}
\end{corollary}
In next theorem, we define the Mellin transforms of the extended Mittag-Leffler function in terms of the Wright hypergeometric function which is defined by (see \cite{Wright1935}-\cite{Wrigt1935})  as:
$$_p\Psi_{q}(z)=\quad_p\Psi_{q}
                                 \left[
                                   \begin{array}{ccc}
                                     (\alpha_i, A_i)_{1,p} &  &  \\
                                      &  & ;z \\
                                     (\beta_j, B_j)_{1,q} &  &  \\
                                   \end{array}
                                 \right]$$
\begin{eqnarray}\label{6}
&=&\sum\limits_{n=0}^{\infty}\frac{\Gamma(\alpha_{1}+A_{1}n)\cdots\Gamma(\alpha_{p}+A_{p}n)}{\Gamma(\beta_{1}+B_{1}n)\cdots
\Gamma(\beta_{q}+B_{q}n)}\frac{z^{n}}{n!}
\end{eqnarray}
where $\beta_{r}$ and $\mu_{s}$  are real positive numbers such that
\begin{eqnarray*}
1+\sum\limits_{s=1}^{q}B_{s}-\sum\limits_{r=1}^{p}A_{r}\geq0.
\end{eqnarray*}
\begin{theorem}\label{thb}
The Mellin transform of extended Mittag-Leffler function is given by
\begin{eqnarray}\label{Mell}
\mathfrak{M}\Big\{E_{\alpha,\beta}^{\gamma,c;\lambda,\rho}(z;p);s\Big\}=\frac{\Gamma^{\lambda;\rho}(s)\Gamma(c+s-\gamma)}{\Gamma(\gamma)\Gamma(c-\gamma)}
._2\Psi_2\left[
         \begin{array}{cc}
           (c,1), (\gamma+s,1), & \\
            & ,z\\
           (\beta,\gamma),(c+2s,1),&  \\
         \end{array}
       \right]
\end{eqnarray}
where $p\geq0$, $\Re(c)>\Re(\gamma)>0$, $\Re(\alpha)>0$,  $\Re(\beta)>0$,  $\Re(\lambda)>0$,  $\Re(\rho)>0$.
\end{theorem}
\begin{proof}
Taking the Mellin transform of extended Mittag-Leffler function defined in (\ref{Fextended}), we have
\begin{eqnarray}\label{Mellin}
\mathfrak{M}\Big\{E_{\alpha,\beta}^{\gamma,c;\lambda,\rho}(z;p);s\Big\}=
\int_{0}^{\infty}p^{s-1}E_{\alpha,\beta}^{\gamma,c;\lambda,\rho}(z;p)dp.
\end{eqnarray}
Using equation (\ref{integral}) in equation (\ref{Mellin}), we have
\begin{align}\label{Mellin1}
\mathfrak{M}\Big\{E_{\alpha,\beta}^{\gamma,c;\lambda,\rho}(z;p);s\Big\}&=\frac{1}{B(\gamma,c-\gamma)}\int_0^\infty p^{s-1}
\Big\{\int_0^1t^{\gamma-1}(1-t)^{c-\gamma-1}._1F_1\Big[\lambda;\rho;-\frac{p}{t(1-t)}\Big]\Big\}\notag\\
&\times E_{\alpha,\beta}^{c}(tz)dtdp.
\end{align}
Interchanging the order of integrations in equation (\ref{Mellin1}), we have
\begin{align}\label{Mellin2}
\mathfrak{M}\Big\{E_{\alpha,\beta}^{\gamma,c;\lambda,\rho}(z;p);s\Big\}&=\frac{1}{B(\gamma,c-\gamma)}\int_0^1\Big[t^{\gamma-1}(1-t)^{c-\gamma-1}
E_{\alpha,\beta}^{c}(tz)\Big]\notag\\
&\times\int_0^\infty p^{s-1}._1F_1\Big[\lambda;\rho;-\frac{p}{t(1-t)}\Big]
dpdt.
\end{align}
Now, taking $u=\frac{p}{t(1-t)}$ in  second integral of equation (\ref{Mellin2}), we get
\begin{eqnarray*}
\int_0^\infty p^{s-1}._1F_1\Big[\lambda;\rho;-\frac{p}{t(1-t)}\Big]
dp &=&\int_0^\infty u^{s-1}t^s(1-t)^s._1F_1\Big[\lambda;\rho;-u\Big]du\\
&=& t^s(1-t)^s\int_0^\infty u^{s-1}._1F_1\Big[\lambda;\rho;-u\Big]du
\end{eqnarray*}
\begin{eqnarray}\label{Mellin3}
&=&t^s(1-t)^s\Gamma^{\lambda,\rho}(s),
\end{eqnarray}
where $\Gamma^{\lambda,\rho}(s)$ is the extended gamma function defined by \cite{Ozergin}.\\
Using equation (\ref{Mellin3}) and the definition of Prabhakar's Mittag-Leffler function in equation (\ref{Mellin2}), we get
$$\mathfrak{M}\Big\{E_{\alpha,\beta}^{\gamma,c;\lambda,\rho}(z;p);s\Big\}$$
\begin{eqnarray*}
&=&\frac{\Gamma^{\lambda,\rho}(s)}{B(\gamma,c-\gamma)}\int_0^1t^{\gamma+s-1}(1-t)^{c+s-\gamma-1}
\sum\limits_{n=0}^{\infty}\frac{(c)_nz^n}{\Gamma(\alpha n+\beta)n!}\int_0^1t^{\gamma+n+s-1}(1-t)^{c+s-\gamma-1}
dt\\
&=&\frac{\Gamma^{\lambda,\rho}(s)}{B(\gamma,c-\gamma)}\sum\limits_{n=0}^{\infty}\frac{(c)_nz^n}{\Gamma(\alpha n+\beta)n!}\frac{\Gamma(\gamma+n+s)\Gamma(c+s-\gamma)}{\Gamma(c+n+2s)}\\
&=&\frac{\Gamma^{\lambda,\rho}(s)\Gamma(c+s-\gamma)}{\Gamma(\gamma)\Gamma(c-\gamma)}
\sum\limits_{n=0}^{\infty}\frac{\Gamma(c+n)z^n}{\Gamma(\alpha n+\beta)n!}\frac{\Gamma(\gamma+n+s)}{\Gamma(c+n+2s)}\\
&=&\frac{\Gamma^{\lambda,\rho}(s)\Gamma(c+s-\gamma)}{\Gamma(\gamma)\Gamma(c-\gamma)}._2\Psi_2\left[
         \begin{array}{cc}
           (c,1), (\gamma+s,1), & \\
            & ,z\\
           (\beta,\gamma),(c+2s,1),&  \\
         \end{array}
       \right]
\end{eqnarray*}
which is the desired proof.
\end{proof}
\begin{corollary}\label{cor1}
Taking $s=1$ in Theorem \ref{thb} and using $\Gamma^{\lambda,\rho}(1)=\frac{\Gamma(\rho)\Gamma(\lambda-1)}{\Gamma(\lambda)\Gamma(\rho-1)}$ (see \cite{Ozergin}) in equation (\ref{Mell}), we get
\begin{eqnarray}
\int_{0}^{\infty}E_{\alpha,\beta}^{\gamma,c;\lambda,\rho}(z;p)dp
=\frac{\Gamma(\rho)\Gamma(\lambda-1)\Gamma(c+1-\gamma)}{\Gamma(\lambda)\Gamma(\rho-1)\Gamma(\gamma)\Gamma(c-\gamma)}
._2\Psi_2\left[
         \begin{array}{cc}
           (c,1), (\gamma+1,1), & \\
            & ,z\\
           (\beta,\gamma),(c+2,1),&  \\
         \end{array}
       \right]
\end{eqnarray}
\end{corollary}
\begin{corollary}
Taking $\lambda=\rho$ in Corollary \ref{cor1}, we get the following result of extended Mittag-Leffler function defined by \cite{OY}
\begin{eqnarray}
\int_{0}^{\infty}E_{\alpha,\beta}^{\gamma,c}(z;p)dp
=\frac{\Gamma(c+1-\gamma)}{\Gamma(\gamma)\Gamma(c-\gamma)}
._2\Psi_2\left[
         \begin{array}{cc}
           (c,1), (\gamma+1,1), & \\
            & ,z\\
           (\beta,\gamma),(c+2,1),&  \\
         \end{array}
       \right]
\end{eqnarray}
\end{corollary}
\begin{corollary}
Taking the inverse Mellin transform on both sides of equation (\ref{Mell}), we get the following elegant complex integral representation
\begin{align}
E_{\alpha,\beta}^{\gamma,c;\lambda,\rho}(z;p)&=\frac{1}{2\pi\iota\Gamma(\gamma)\Gamma(c-\gamma)}
\int_{\nu-\iota\infty}^{\nu+\iota\infty}\Gamma^{\lambda,\rho}(s)\Gamma(c+s-\gamma)\notag\\
&\times{}_2\Psi_2\left[
         \begin{array}{cc}
           (c,1), (\gamma+s,1), & \\
            & ,z\\
           (\beta,\gamma),(c+2s,1),&  \\
         \end{array}
       \right]p^{-s}ds,
\end{align}
where $\nu>0$.
\end{corollary}
%%%%%%%%%%%%%%%%%%%%%%%%%%%%%%%%%%%%%%%%%%%%%%%%%%%%%%%%%%%%%%%%%%%%%%%%%%%%%%%%%%%%%%%%%%%%%%%
\section{Derivative properties of extended Mittag-Leffler function }
In this section, we define further extension of extended  Riemann-Liouville fractional derivative, extended fractional derivative formula of further extended Mittag-Leffler function and some derivative properties of extended Mittag-Leffler function.
\begin{definition}\label{defa}
The well-known Riemann-Liouville fractional derivative of order $\mu$ is defined by
\begin{eqnarray}
\mathfrak{D}_{x}^{\mu}=\frac{1}{\Gamma(-\mu)}\int_0^xf(t)(x-t)^{-\mu-1}dt, \Re(\mu)>0.
\end{eqnarray}
For the case $m-1<\Re(\mu)<m$ where $m=1,2,\cdots$, it follows
\begin{eqnarray*}
\mathfrak{D}_{x}^{\mu}=\frac{d^m}{dx^m}\mathfrak{D}_{x}^{\mu-m}\Big\{f(z)\Big\}
\end{eqnarray*}
\begin{eqnarray}
=\frac{d^m}{dx^m}\Big\{\frac{1}{\Gamma(-\mu+m)}\int_0^xf(t)(x-t)^{-\mu+m-1}dt\Big\}, \Re(\mu)>0.
\end{eqnarray}
\end{definition}
\begin{definition}\label{defb}(see \cite{OY})
The extended Riemann-Liouville fractional derivative of order $\mu$ is defined by
\begin{eqnarray}
\mathfrak{D}_{x}^{\mu,p}=\frac{1}{\Gamma(-\mu)}\int_0^xf(t)(x-t)^{-\mu-1}\exp\Big(-\frac{px^2}{t(x-t)}\Big) dt, \Re(\mu)>0.
\end{eqnarray}
For the case $m-1<\Re(\mu)<m$ where $m=1,2,\cdots$, it follows
\begin{eqnarray*}
\mathfrak{D}_{x}^{\mu,p}=\frac{d^m}{dx^m}\mathfrak{D}_{x}^{\mu-m,p}\Big\{f(z)\Big\}
\end{eqnarray*}
\begin{eqnarray}
=\frac{d^m}{dx^m}\Big\{\frac{1}{\Gamma(-\mu+m)}\int_0^xf(t)(x-t)^{-\mu+m-1}\exp\Big(-\frac{px^2}{t(x-t)}\Big) dt\Big\}, \Re(\mu)>0.
\end{eqnarray}
\end{definition}

\begin{definition}\label{defc}
Here, we define the extension of extended Riemann-Liouville fractional derivative of order $\mu$ as
\begin{eqnarray}\label{Efrac}
\mathfrak{D}_{x}^{\mu,p}=\frac{1}{\Gamma(-\mu)}\int_0^xf(t)(x-t)^{-\mu-1}.
_1F_1\Big[\lambda;\rho;\Big(-\frac{px^2}{t(x-t)}\Big)\Big] dt, \Re(\mu)>0.
\end{eqnarray}
For the case $m-1<\Re(\mu)<m$ where $m=1,2,\cdots$, it follows
\begin{eqnarray*}
\mathfrak{D}_{x}^{\mu,p}=\frac{d^m}{dx^m}\mathfrak{D}_{x}^{\mu-m,p}\Big\{f(z)\Big\}
\end{eqnarray*}
\begin{eqnarray}
=\frac{d^m}{dx^m}\Big\{\frac{1}{\Gamma(-\mu+m)}\int_0^xf(t)(x-t)^{-\mu+m-1}
._1F_1\Big[\lambda;\rho;\Big(-\frac{px^2}{t(x-t)}\Big)\Big] dt\Big\}, \Re(\mu)>0.
\end{eqnarray}
\end{definition}
Obviously, if $\lambda=\rho$, then definition \ref{defc} reduces to extended fractional derivative \ref{defb}. Similarly, if $\lambda=\rho$ and $p=0$, then definition \ref{defc} reduces to the well-known Riemann-Liouville fractional derivative \ref{defa}.
\begin{theorem}
Let $p\geq0$, $\Re(\mu)>\Re(\delta)>0$, $\Re(\alpha)>0$, $\Re(\beta)>0$. Then
\begin{eqnarray}
\mathfrak{D}_{z}^{\delta-\mu,p}\Big\{z^{\delta-1}E_{\alpha,\beta}^{c}(z)\Big\}=\frac{z^{\mu-1}B(\delta,c-\delta)}
{\Gamma(\mu-\delta)}E_{\alpha,\beta}^{\delta,\mu;\lambda,\rho}(z;p)
\end{eqnarray} 
\end{theorem}
\begin{proof}
Replacing $\mu$ by $\delta-\mu$ in the definition of extension of extended fractional derivative formula (\ref{Efrac}), we have
$$\mathfrak{D}_{z}^{\delta-\mu,p}\Big\{z^{\delta-1}E_{\alpha,\beta}^{c}(z)\Big\}$$
\begin{eqnarray*}
&=&\frac{1}{\Gamma(\mu-\delta)}\int_0^zt^{\delta-1}E_{\alpha,\beta}^{c}(t)(z-t)^{-\delta+\mu-1}.
_1F_1\Big[\lambda;\rho;\Big(-\frac{pz^2}{t(z-t)}\Big)\Big] dt\\
&=&\frac{z^{-\delta+\mu-1}}{\Gamma(\mu-\delta)}\int_0^zt^{\delta-1}E_{\alpha,\beta}^{c}(t)(1-\frac{t}{z})^{-\delta+\mu-1}.
_1F_1\Big[\lambda;\rho;\Big(-\frac{pz^2}{t(z-t)}\Big)\Big] dt
\end{eqnarray*}
Taking $u=\frac{t}{z}$ in above equation, we have
$$\mathfrak{D}_{z}^{\delta-\mu,p}\Big\{z^{\delta-1}E_{\alpha,\beta}^{c}(z)\Big\}$$
\begin{eqnarray}\label{Frac}
&=&\frac{z^{\mu-1}}{\Gamma(\mu-\delta)}\int_0^z u^{\delta-1}(1-u)^{-\delta+\mu-1}.
_1F_1\Big[\lambda;\rho;\Big(-\frac{p}{u(1-u)}\Big)\Big] E_{\alpha,\beta}^{c}(uz)du.
\end{eqnarray}
Comparing equation (\ref{Frac}) with equation (\ref{integral}), we get the desired result.
\end{proof}
In the following theorem we define the derivative properties of extended Mittag-Leffler function.
\begin{theorem}
For the extended Mittag-Leffler function, we have the following derivative formula:
\begin{eqnarray}\label{der}
\frac{d^n}{dz^n}\Big\{E_{\alpha,\beta}^{\gamma,c;\lambda,\rho}(z;p)\Big\}
=\frac{(c)_n(\lambda)_n}{(\rho)_n}E_{\alpha,\beta+n\alpha}^{\gamma+n,c+ n;\lambda+n,\rho+n}(z;p).
\end{eqnarray}
\end{theorem}
\begin{proof}
Taking derivative of equation (\ref{Fextended}) with respect $z$, we have
\begin{eqnarray}\label{der1}
\frac{d}{dz}\Big\{E_{\alpha,\beta}^{\gamma,c;\lambda,\rho}(z;p)\Big\}
=\frac{c\lambda}{\rho}E_{\alpha,\beta+\alpha}^{\gamma+1,c+ 1;\lambda+1,\rho+1}(z;p).
\end{eqnarray}
Again taking derivative of equation (\ref{der1}), with respect to $z$, we have
\begin{eqnarray}
\frac{d^2}{dz^2}\Big\{E_{\alpha,\beta}^{\gamma,c;\lambda,\rho}(z;p)\Big\}
=\frac{c(c+1)\lambda(\lambda+1)}{\rho(\rho+1)}E_{\alpha,\beta+2\alpha}^{\gamma+2,c+ 2;\lambda+2,\rho+2}(z;p).
\end{eqnarray}
Continuing in this way up to $n$, we get the required result. 
\end{proof}
\begin{theorem}
The following differentiation formula holds for the extended Mittag-Leffler function
\begin{eqnarray}
\frac{d^n}{dz^n}\Big\{z^{\beta-1}E_{\alpha,\beta}^{\gamma,c;\lambda,\rho}(\mu z^\alpha;p)\Big\}
=z^{\beta-n-1}E_{\alpha,\beta-n}^{\gamma+n,c+ n;\lambda+n,\rho+n}(\mu z^\alpha;p).
\end{eqnarray}
\end{theorem}
\begin{proof}
In equation (\ref{der}), replace $z$ by $\mu z^\alpha$ and multiply by $z^{\beta-1}$ and then taking $nth$ derivative with respect to $z$, we get the required result. 
\end{proof}

%%%%%%%%%%%%%%%%%%%%%%%%%%%%%%%%%%%%%%%%%%%%%%%%%%%%%%
\section{conclusion}
In this paper, we established further extension of extended Mittag-Leffler recently introduced by \cite{OY}. We conclude that if $\lambda=\rho$, then we get the results of exytended Mittag-Leffler function.\\ \\

%\noindent{\bf Conflict of Interests:}
%\noindent We  declare  that there is no conflict of interests regarding the publication of this manuscript.\vskip 2mm
%\noindent{\bf Authors’ contributions:}\vskip 2mm
%
%\noindent All authors contributed equally to the written of this manuscript. All authors read and approved the final manuscript.\\
%\textbf{Acknowledgements:}\\
%The authors would like to express profound gratitude to referees for his/her deeper
%review of this paper and the referee's useful suggestions that led to an improved presentation of the paper.\\

\end{document}